\RequirePackage[l2tabu, orthodox]{nag}
\documentclass{amsart}
\usepackage{etoolbox}
\usepackage[hypertexnames=false]{hyperref}
\usepackage{microtype}
\usepackage{amssymb,dsfont}
\usepackage{breqn}
\usepackage[nobreak]{cite}
\setkeys{breqn}{compact}
\usepackage[utf8]{inputenc}
\usepackage[T1]{fontenc}
\usepackage{lmodern}
\usepackage{xcolor}
\usepackage{enumitem}
\usepackage[capitalize,nameinlink]{cleveref}
\usepackage[english,polish]{babel}
\crefname{empty}{}{}
\newlist{alist}{enumerate}{1}
\setlist[alist]{label=(\alph*),itemsep=.1ex, ref=(\alph*)}
\crefalias{alisti}{empty}
\newlist{rlist}{enumerate}{1}
\setlist[rlist]{label=(\roman*),itemsep=.1ex, ref=(\roman*)}
\crefalias{rlisti}{empty}
\newlist{nlist}{enumerate}{1}
\setlist[nlist]{label=(\arabic*),itemsep=.1ex, ref=(\arabic*)}
\crefalias{nlisti}{empty}
\crefname{equation}{}{}
\Crefname{figure}{Figure}{Figures}
\crefname{page}{page}{pages}
\Crefname{enumi}{}{}
\Crefname{subsection}{Subsection}{Subsections}
\def\theoremname{Theorem}%
\def\propositionname{Proposition}%
\def\lemmaname{Lemma}%
\def\corollaryname{Corollary}%
\def\definitionname{Definition}%
\def\conventionname{Convention}%
\def\axiomname{Axiom}%
\def\remarkname{Remark}%
\def\examplename{Example}%
\def\questionname{Question}%
\def\constructionname{Construction}%
\def\notationname{Notation}%
\def\assumptionname{Assumption}%
\def\conjecturename{Conjecture}%
\newtheorem{thm}{\theoremname}[section]
\theoremstyle{plain}
\newtheorem{theorem}[thm]{\theoremname}

\newtheorem{prop}[thm]{Proposition}
\newtheorem{lemma}[thm]{\lemmaname}
\newtheorem{corollary}[thm]{\corollaryname}
\theoremstyle{definition}

\newtheorem{example}[thm]{\examplename}

\DeclareMathOperator{\sign}{sgn}
\DeclareMathOperator{\supp}{supp}

\newcommand{\field}[1]{\ensuremath{\mathbb{#1}}}

\newcommand{\0}{\ensuremath{\mathbf{0}}}
\newcommand{\1}{\ensuremath{\mathbf{1}}}

\newcommand{\ignoreSpellCheck}[1]{#1}

\begin{document}

\title[On a problem of J.~Matkowski and J.~Wesołowski,~II]{On a problem of Janusz Matkowski and Jacek Wesołowski,~II}

\selectlanguage{polish}
\author[J. Morawiec]{Janusz Morawiec}
\address{Instytut Matematyki{}\\
Uniwersytet Śląski{}\\
Bankowa 14, PL-40-007 Katowice{}\\
Poland}
\email{morawiec@math.us.edu.pl}

\author[T. Zürcher]{Thomas Zürcher}
\address{Instytut Matematyki{}\\
Uniwersytet Śląski{}\\
Bankowa 14, PL-40-007 Katowice{}\\
Poland}
\email{thomas.zurcher@us.edu.pl}
\selectlanguage{English}

\begin{abstract}
We continue our study started in \cite{MZ} of the
functional equation
\begin{equation*}
\varphi(x)=\sum_{n=0}^{N}\varphi(f_n(x))-\sum_{n=0}^{N}\varphi(f_n(0))
\end{equation*}
and its increasing and continuous solutions $\varphi\colon[0,1]\to[0,1]$
such that $\varphi(0)=0$ and $\varphi(1)=1$. In this paper we assume that $f_0,\ldots,f_N\colon[0,1]\to[0,1]$ are strictly increasing contractions such that
\begin{equation*}
0\leq f_0(0)<f_0(1)\leq f_1(0)<\cdots <f_{N-1}(1)\leq f_N(0)<f_N(1)\leq 1
\end{equation*}
 and at least one of the weak inequalities is strong.
\end{abstract}

\maketitle


\section{Introduction}\label{introduction}
Fix $N\in\mathbb N$ and strictly increasing contractions $f_0,\ldots,f_N\colon[0,1]\to[0,1]$ such that
\begin{equation}\label{0<1}
0\leq f_0(0)<f_0(1)\leq f_1(0)<\cdots<f_{N-1}(1)\leq f_N(0)<f_N(1)\leq 1.
\end{equation}
We continue our study of the existence of solutions $\varphi$ of the functional equation
\begin{equation}\label{e}
\varphi(x)=\sum_{n=0}^{N}\varphi(f_n(x))-\sum_{n=0}^{N}\varphi(f_n(0))
\tag{$\textsf{E}$}
\end{equation}
in the class $\mathcal C$ consisting of all increasing and continuous functions  $\varphi\colon [0,1]\to [0,1]$ satisfying the following boundary conditions
\begin{equation}\label{cond}
\varphi(0)=0\hspace{3ex}\hbox{ and }\hspace{3ex}\varphi(1)=1.
\end{equation}
In this paper we assume, in contrast to~\cite{MZ}, that
\begin{equation}\label{[0,1]}
\bigcup_{n=0}^{N}\big[f_n(0),f_n(1)\big]\neq [0,1].
\end{equation}


\section{Preliminaries}\label{Preliminaries}
Throughout this paper for all $k\in\mathbb N$ and $n_1,\ldots,n_k\in\{0,\ldots,N\}$ we denote the composition $f_{n_1}\circ\cdots\circ f_{n_k}$ by $f_{n_1,\ldots,n_k}$. Moreover, we extend the notation to the case $k=0$ by letting $f_{n_1,\ldots,n_0}$ be the identity.

We begin with three lemmas. The proof of the first one is very easy, so we omit it.

\begin{lemma}\label{lem21}
Fix $m\in\mathbb N$ and nonnegative real numbers $\alpha_1,\ldots,\alpha_m$ such that $\sum_{i=1}^m\alpha_i=1$. If $\varphi_1,\ldots,\varphi_m\in\mathcal C$, then
$\sum_{i=1}^m\alpha_i\varphi_i\in\mathcal C$.
\end{lemma}

\begin{lemma}\label{lem22}
If $\varphi\in\mathcal C$, then $\varphi(f_0(0))=0$ and $\varphi(f_N(1))=1$.
\end{lemma}

\begin{proof}
By \eqref{cond}, \eqref{e}, \eqref{0<1}, and the monotonicity of $\varphi$ we have
\begin{equation*}
\begin{split}
1&=\varphi(1)=\varphi(f_N(1))-\varphi(f_0(0))
+\sum_{n=1}^{N}\big[\varphi(f_{n-1}(1))-\varphi(f_n(0))\big]\\
&\leq\varphi(f_N(1))-\varphi(f_0(0)).
\end{split}
\end{equation*}
As the image of~$[0,1]$ under~$\varphi$ lies in~$[0,1]$, we infer that $\varphi(f_0(0))=0$ and $\varphi(f_N(1))=1$.
\end{proof}

Now we want to show that if all the contractions $f_0,\ldots, f_N$ are \emph{nonsingular} (i.e.\ $f_0^{-1}(A),\ldots,f_N^{-1}(A)$ have Lebesgue measure zero for every set $A\subset [0,1]$ of Lebesgue measure zero\footnote{See~\cite{LM1994}. Note also that as the inverses of the contractions exist and are continuous and increasing, being nonsingular is equivalent to the inverses being absolutely continuous, see for example Theorem~7.1.38 in \cite{KaKr96}.}), then the class $\mathcal C$ is determined by two of its subclasses $\mathcal C_a$ and $\mathcal C_s$ of all absolutely continuous and all singular functions, respectively. Repeating directly the proof of Remark~2.2 from~\cite{MZ} with the use of \cref{lem22} we get the following result.

\begin{lemma}\label{lem23}
Assume that all the contractions $f_0,\ldots, f_N$ are nonsingular. Then, both the absolutely continuous and the singular parts\footnote{The parts are unique up to a constant. For definiteness, we choose them such that both of them map $0$~to~$0$.} of every element from $\mathcal C$ satisfy $\eqref{e}$ for every $x\in[0,1]$. 	 	
\end{lemma}

By the monotonicity of $f_0$ and $f_N$, it is easy to prove that the sequence $(f_{\scriptstyle\underbrace{0,\ldots,0}_{k}}(0))_{k\in\mathbb N}$ is increasing and the sequence
$(f_{\scriptstyle\underbrace{N,\ldots,N}_{k}}(1))_{k\in\mathbb N}$ is decreasing. Hence both are convergent. Put
\begin{equation*}\label{endpoints}
\0=\lim_{k\to\infty}f_{\scriptstyle\underbrace{0,\ldots,0}_{k}}(0)
\quad\hbox{ and }\quad
\1=\lim_{k\to\infty}f_{\scriptstyle\underbrace{N,\ldots,N}_{k}}(1).
\end{equation*}

It is clear that $\0$ is the unique fixed point of $f_0$ and $\1$ is the  unique fixed point of $f_N$, i.e.
\begin{equation}\label{fixed}
f_0(\0)=\0\quad\hbox{ and }\quad f_N(\1)=\1.
\end{equation}
Moreover,
\begin{equation*}
\1=\lim_{k\to\infty}f_{\scriptstyle\underbrace{N,\ldots,N}_{k}}(0),
\end{equation*}
because for every $k\in\mathbb N$ we have $|f_{\scriptstyle\underbrace{N,\ldots,N}_{k}}(1)-f_{\scriptstyle\underbrace{N,\ldots,N}_{k}}(0)|\leq c^k$, where $c\in(0,1)$ is a Lipschitz constant of $f_N$.

\begin{lemma}\label{lem24}
Assume that $\varphi\in\mathcal C$. Then $\varphi(\0)=0$ and $\varphi(\1)=1$.
\end{lemma}

\begin{proof}
We first prove that $\varphi(\0)=0$.

By \cref{lem22} we have $\varphi(f_0(0))=0$. Fix $k\in\mathbb N$ and assume inductively that $\varphi(f_{\scriptstyle \underbrace{0,\ldots,0}_{k}}(0))=0$. Applying the induction hypothesis, \eqref{e}, \mbox{\vphantom{g}\cref{lem22}} and the monotonicity of $f_0,\ldots,f_N$ and $\varphi$, we get
\begin{equation*}
\begin{split}
0&=\varphi(f_{\scriptstyle \underbrace{0,\ldots,0}_{k}}(0))
=\sum_{n=0}^{N}\varphi(f_n(f_{\scriptstyle \underbrace{0,\ldots,0}_{k}}(0)))-\sum_{n=0}^{N}\varphi(f_n(0))\\
&=\varphi(f_{\scriptstyle \underbrace{0,\ldots,0}_{k+1}}(0))
+\sum_{n=1}^{N}\varphi(f_n(f_{\scriptstyle \underbrace{0,\ldots,0}_{k}}(0)))
-\sum_{n=1}^{N}\varphi(f_n(0))\\
&\geq\varphi(f_{\scriptstyle \underbrace{0,\ldots,0}_{k+1}}(0))\geq 0.
\end{split}
\end{equation*}
Hence $\varphi(f_{\scriptstyle \underbrace{0,\ldots,0}_{k+1}}(0))=0$.
Now the continuity of $\varphi$ gives
\begin{equation*}
\varphi(\0)=\lim_{k\to\infty}\varphi(f_{\scriptstyle \underbrace{0,\ldots,0}_{k}}(0))=0.
\end{equation*}

To prove that $\varphi(\1)=1$ observe first that by \eqref{0<1} and the monotonicity of $\varphi$ we have $\varphi(f_n(1))\leq\varphi(f_{n+1}(0))$ for every $n\in\{0,\ldots,N-1\}$. We want to show that
\begin{equation}\label{i2}
\varphi(f_n(1))=\varphi(f_{n+1}(0))
\end{equation}
for every $n\in\{0,\ldots,N-1\}$. Suppose that, contrary to our claim, there exists $n\in\{0,\ldots,N-1\}$ such that $\varphi(f_n(1))<\varphi(f_{n+1}(0))$. Then, using \cref{lem22} and arguing as in its proof, we obtain
\begin{equation*}
\begin{split}
1&=\varphi(1)=
\sum_{n=0}^{N-1}\varphi(f_n(1))+1-\sum_{n=0}^{N}\varphi(f_n(0))\\
&<\sum_{n=0}^{N-1}\varphi(f_{n+1}(0))+1-\sum_{n=0}^{N}\varphi(f_n(0))
=1-\varphi(f_0(0))=1,
\end{split}
\end{equation*}
a contradiction.

Now we show by induction that
\begin{equation}\label{i3}
\varphi(f_{\scriptstyle \underbrace{N,\ldots,N}_{k}}(1))=1
\end{equation}
for all $k\in\mathbb N$.
The first step of the induction holds due to \cref{lem22}. Fix $k\in\mathbb N$ and assume that \eqref{i3} holds. Then applying \eqref{i3}, \eqref{e}, \cref{lem22}, \eqref{i2} and the monotonicity of $f_0,\ldots,f_N$ and $\varphi$ we get
\begin{equation*}
\begin{split}
1&=\varphi(f_{\scriptstyle \underbrace{N,\ldots,N}_{k}}(1))=
\sum_{n=0}^{N}\varphi(f_n(f_{\scriptstyle \underbrace{N,\ldots,N}_{k}}(1)))-\sum_{n=1}^{N}\varphi(f_n(0))\\
&\leq \sum_{n=0}^{N-1}\varphi(f_{n}(f_{\scriptstyle \underbrace{N,\ldots,N}_{k}}(1)))+\varphi(f_{\scriptstyle \underbrace{N,\ldots,N}_{k+1}}(1))-
\sum_{n=0}^{N-1}\varphi(f_n(1))\\
&\leq\varphi(f_{\scriptstyle \underbrace{N,\ldots,N}_{k+1}}(1))\leq 1.
\end{split}
\end{equation*}
Hence $\varphi(f_{\scriptstyle \underbrace{N,\ldots,N}_{k+1}}(1))=1$.
Finally, passing with $k$ to infinity in~\eqref{i3} and using the continuity of $\varphi$ we obtain $\varphi(\1)=1$.
\end{proof}


\section{Basic property of solutions}\label{Basic}
Define recursively a sequence $(A_k)_{k\in\mathbb N}$ of subsets of the interval $[0,1]$ as follows:
\begin{equation*}
A_0=[0,1]
\quad\hbox{ and }\quad
A_k=\bigcup_{n=0}^{N}f_n(A_{k-1})\quad\hbox{for every }k\in\mathbb N.
\end{equation*}
By~\eqref{[0,1]} we have $A_1=\bigcup_{n=0}^{N}\big[f_n(0),f_n(1)\big]\varsubsetneq A_0$. Moreover, a witness of the strict inclusion can be found that is different from $0$~and~$1$. This jointly with an easy induction shows that $A_{k+1}\varsubsetneq A_k$ for every $k\in\mathbb N$. Again there is a witness of the strict inequality differing from $0$~and~$1$. Put
\begin{equation*}
A_*=\bigcap_{k\in\mathbb N}A_k.
\end{equation*}
It is clear that $A_*$ is compact and
\begin{equation}\label{structure of invariant set}
A_*=\bigcup_{n=0}^{N}f_n(A_*).
\end{equation}
We will show that the just constructed set $A_*$, called the \emph{attractor} of the iterated function system $\{f_0,\ldots,f_N\}$ (see~\cite{B1988}), is
a \emph{Cantor-like set}, i.e.\ uncountable, nowhere dense and perfect subset of $\mathbb R$ (see~\cite{WiHa1993}); note that $A_*$ is uncountable and nowhere dense, which follows from its construction. Moreover, we will see in \cref{perfect} that $A_*$ is perfect and in \cref{ex34} that it is of Lebesgue measure zero if $f_0,\ldots,f_N$ are similitudes, whereas in the general case it can happen that $A_*$ is of positive Lebesgue measure (see \cite{MZPositive}).

From the construction we have
\begin{equation*}
A_*=\bigcap_{k\in\mathbb N}\left(\bigcup_{n_1,\ldots,n_k\in\{0,\ldots,N\}}
\big[f_{n_1,\ldots,n_k}(0),f_{n_1,\ldots,n_k}(1)\big]\right).
\end{equation*}
Whenever a point $x$ can be written as
\begin{equation}\label{x}
x=\lim_{k\to\infty}f_{x_1,\ldots,x_k}(0)=\lim_{k\to\infty}f_{x_1,\ldots,x_k}(1),
\end{equation}
we say that $x$ has an \emph{address}\footnote{We have come across the term \emph{coding} as well.} (see~\cite{B1988}).

\begin{lemma}\label{attractor and addresses}
The set~$A_*$ is exactly the set of points in~$[0,1]$ that have an address.
\end{lemma}

\begin{proof}
Let $x\in A_*$. Note that for every $k\in \field{N}$ there exist $x_1^k,\ldots,x_k^k\in \{0,\ldots,N\}$ such that $x\in[f_{x_1^k,\ldots,x_k^k}(0),f_{x_1^k,\ldots,x_k^k}(1)]$ with $x_m^n$ not necessarily agreeing with $x_m^l$ for different $l$ and $n$, however, as each $x_m^l$ is chosen from the finite set $\{0,\ldots,N\}$, we may apply a Cantor diagonal argument to get a sequence as wished.

It is easy to see that every sequence $(x_k)_{k\in\mathbb N}$ of elements of $\{0,\ldots,N\}$ is an address of a point from the set $A_*$.
\end{proof}

Note that
\begin{equation}\label{01inA}
\0=\min A_*\quad\hbox{ and }\quad \1=\max A_*.
\end{equation}
Since $A_*$ is a closed set, it follows that $[\0,\1]\setminus A_*$ is an open set. Moreover,
\begin{equation}\label{A*}
[\0,\1]\setminus A_*=\bigcup_{k\in\mathbb N}\mathop{\bigcup_{0\leq n_1,\ldots,n_{k-1}\leq N}}_{0\leq n_k\leq N-1}\big(f_{n_1,\ldots,n_k}(\1),f_{n_1,\ldots,n_k+1}(\0)\big)
\end{equation}
and for all $k\in\mathbb N$, $n_1,\ldots,n_{k-1}\in\{0,\ldots,N\}$ and $n_k\in\{0,\ldots,N-1\}$ the interval $(f_{n_1,\ldots,n_k}(\1),f_{n_1,\ldots,n_k+1}(\0))$ is a connected component of the set $[\0,\1]\setminus A_*$.

Now we are in a position to show that any $\varphi\in\mathcal C$ is constant on the closure of each connected component of the set $[0,1]\setminus A_*$. We do it in two steps.

\begin{lemma}\label{lem31}
Assume that $\varphi\in\mathcal C$. Then:
\begin{enumerate}
\item\label{restriction phi to 0} $\varphi|_{[0,\0]}=0$;
\item\label{restriction phi to 1} $\varphi|_{[\1,1]}=1$;
\item\label{phi constant} $\varphi|_{[f_n(\1),f_{n+1}(\0)]}$ is constant for every $n\in\{0,\ldots,N-1\}$.
\end{enumerate}
\end{lemma}

\begin{proof}
To prove (\ref{restriction phi to 0})~and~(\ref{restriction phi to 1}) it is enough to apply \cref{lem24} jointly with the monotonicity of $\varphi$.

Let us tackle (\ref{phi constant}). According to~\eqref{i2} and to the monotonicity of $f_0,\ldots,f_N$ and $\varphi$, we see that $\varphi(f_n(\1))\leq\varphi(f_{n+1}(\0))$.
Suppose that, contrary to our claim, there exists $n\in\{0,\ldots,N-1\}$ such that $\varphi(f_n(\1))<\varphi(f_{n+1}(\0))$. Then, using \cref{lem24}, \eqref{e}, and the first equality of \eqref{fixed} we get
\begin{equation*}
\begin{split}
1&=\varphi(\1)=
\sum_{n=0}^{N-1}\varphi(f_n(\1))+\varphi(f_N(\1))-\sum_{n=0}^{N}\varphi(f_n(0))\\
&<\sum_{n=0}^{N-1}\varphi(f_{n+1}(\0))+1-\sum_{n=0}^{N}\varphi(f_n(0))
=\varphi(\0)-\varphi(f_0(\0))+1=1,
\end{split}
\end{equation*}
a contradiction.
\end{proof}

\begin{lemma}\label{lem32}
Assume that $\varphi\in\mathcal C$. Then for all $k\in\mathbb N$, $n_1,\ldots,n_{k-1}\in\{0,\ldots,N\}$ and $n_k\in\{0,\ldots,N-1\}$ there exists  $c_{n_1,\ldots,n_k}\in[0,1]$ such that
\begin{equation}\label{const}
\varphi|_{[f_{n_1,\ldots,n_k}(\1),f_{n_1,\ldots,n_k+1}(\0)]}=c_{n_1,\ldots,n_k}.
\end{equation}
\end{lemma}

\begin{proof}
We proceed by induction on $k$.

The first step of the induction is implied by assertion~(\ref{phi constant}) of \cref{lem31}.

Fix $k\in\mathbb N$, $n_1,\ldots,n_{k-1}\in\{0,\ldots,N\}$, $n_k\in\{0,\ldots,N-1\}$ and assume that there exists $c_{n_1,\ldots,n_k}\in[0,1]$ such that \eqref{const} holds.
Then \eqref{const}, \eqref{e} and the monotonicity of $f_0,\ldots,f_N$ and $\varphi$ imply
\begin{equation*}
\begin{split}
c_{n_1,\ldots,n_k}&=\varphi(f_{n_1,\ldots,n_k}(\1))
=\sum_{n=0}^{N}\varphi(f_{n,n_1,\ldots,n_k}(\1))-\sum_{n=0}^{N}\varphi(f_n(0))\\
&\leq\sum_{n=0}^{N}\varphi(f_{n,n_1,\ldots,n_k+1}(\0))-
\sum_{n=0}^{N}\varphi(f_n(0))=\varphi(f_{n_1,\ldots,n_k+1}(\0))\\
&=c_{n_1,\ldots,n_k}.
\end{split}
\end{equation*}
Hence
\begin{equation*}
\sum_{n=0}^{N}\varphi(f_{n,n_1,\ldots,n_k}(\1))=
\sum_{n=0}^{N}\varphi(f_{n,n_1,\ldots,n_k+1}(\0)),
\end{equation*}
and applying again the monotonicity of $f_0,\ldots,f_N$ and $\varphi$, we obtain
\begin{equation*}
\varphi(f_{n,n_1,\ldots,n_k}(\1))=\varphi(f_{n,n_1,\ldots,n_k+1}(\0))
\end{equation*}
for every $n\in\{0,\ldots,N\}$.
\end{proof}

Combining \cref{lem31,lem32} with \eqref{A*}, we get the following result.

\begin{theorem}\label{thm33}
If the set $A_*$ has Lebesgue measure zero, then $\mathcal C=\mathcal C_s$.
\end{theorem}

We now give an example of contractions $f_0,\ldots, f_N$ for which the set $A_*$ is of  Lebesgue measure zero.

\begin{example}\label{ex34}
Assume additionally to our assumptions in the introduction that $f_0,\ldots,f_N$ are similitudes, i.e.\
\begin{equation*}
f_n(x)=(\beta_n-\alpha_n)x+\alpha_n
\end{equation*}
for all $x\in[0,1]$ and $n\in\{0,\ldots,N\}$, where
\begin{equation*}
0\leq\alpha_0<\beta_0\leq\alpha_1<\beta_1\leq\cdots\leq\alpha_N<\beta_N\leq 1
\quad\hbox{and}\quad
\bigcup_{n=0}^{N}\big[\alpha_n,\beta_n\big]\neq [0,1].
\end{equation*}
Clearly, \eqref{0<1} and \eqref{[0,1]} hold. Denote by $l$ the Lebesgue measure on the real line and put $d=l(A_0\setminus A_1)$. By a simple induction we get $l(A_k\setminus A_{k+1})=d(1-d)^k$ for every $k\in\mathbb N$. From \eqref{0<1} and \eqref{[0,1]} we infer that $d\in(0,1)$ and hence that
\begin{equation*}
l(A_*)=1-\sum_{k=0}^{\infty}l(A_k\setminus A_{k+1})=1-\frac{d}{1-(1-d)}=0.
\end{equation*}
\end{example}

We finish this section with one more property of the set $A_*$.

\begin{theorem}\label{perfect}
The set $A_*$ is perfect.
\end{theorem}

\begin{proof}
We know from its definition that $A_*$ is closed, and it is nonempty by \eqref{01inA}.

Let $x\in A_*$ and fix an address of $x$, i.e.\ a sequence $(x_k)_{k\in\mathbb N}$ of elements of $\{0,\ldots,N\}$ satisfying~\eqref{x}; we can choose such a sequence according to \cref{attractor and addresses}. To complete the proof, we need to show that in each neighbourhood of~$x$ we can find some element belonging to $A_*\setminus\{x\}$.

Fix $\varepsilon>0$ and $m\in \mathbb N$ so large that $L^{m-1}<\varepsilon$, where $L\in(0,1)$ is the largest Lipschitz constant of the given contractions
$f_0,\ldots,f_N$. Define a sequence $(y_k)_{k\in\mathbb N}$ by putting
$y_k=x_k$ for all $k\neq m$ and choosing arbitrarily $y_m\in\{0,\ldots,N\}\setminus\{x_m\}$. Then
\begin{equation*}
y=\lim_{k\to\infty}f_{y_1,\ldots,y_k}(0)\in A_*.
\end{equation*}
Since all considered contractions are injective and the addresses of points $x$ and $y$ differ only in the $m$-th coordinate, it follows that $y\neq x$. Moreover,
\begin{equation*}
\begin{split}
|x-y|&=\lim_{k\to\infty}|f_{x_1,\ldots,x_{m-1}}(f_{x_m,\ldots,x_k}(0))-f_{x_1,\ldots,x_{m-1}}(f_{y_m,\ldots,y_k}(0))|\\
&\leq L^{m-1}\lim_{k\to\infty}|f_{x_m,\ldots,x_k}(0)-f_{y_m,\ldots,y_k}(0)|\leq
L^{m-1}<\varepsilon.
\end{split}
\end{equation*}
The proof is complete.
\end{proof}


\section{Existence of solutions}\label{Existence}
In the previous section we have discussed the behaviour of functions belonging to the class $\mathcal C$, but, up to now, we do not know if $\mathcal C$ contains any function at all.
In this section, we want to show that $\mathcal C\neq\emptyset$.

Fix positive real numbers $p_0,\ldots,p_N$ such that
\begin{equation}\label{prob}
\sum_{n=0}^{N}p_n=1.
\end{equation}
Then there exists a unique Borel probability measure $\mu$ such that
\begin{equation}\label{inv}
\mu(B)=\sum_{n=0}^{N}p_n\mu(f_n^{-1}(B))
\end{equation}
for every Borel set $B\subset [0,1]$ (see \cite{H1981}; cf.\ \cite{F1997}).
From now on the letter $\mu$ will be reserved for the unique Borel probability measure satisfying~\eqref{inv} for every Borel set $B\subset [0,1]$.

Now we are interested in some properties of the measure $\mu$ that will be needed later. We begin with a well-known folklore lemma; for its proof the reader can consult \cite{LaMy1994}.

\begin{lemma}\label{lem41}
The measure $\mu$ is either singular or absolutely continuous with respect to the Lebesgue measure on $\mathbb R$.
\end{lemma}

To formulate the next lemma, which is also well-known (see e.g.\ \cite{B1988}), we recall that the support of the measure $\mu$ is the set $\supp\mu$ of all points $x\in[0,1]$ such that $\mu([x-\varepsilon,x+\varepsilon])>0$ for every $\varepsilon>0$.

\begin{lemma}\label{lem42}
We have $\supp\mu=A_*$. In particular, $\mu([0,1]\setminus A_*)=0$.
\end{lemma}

\begin{lemma}\label{lem43}
The measure $\mu$ is continuous.
\end{lemma}

\begin{proof}
To prove that $\mu$ is continuous it is enough to show $\mu(\{x\})=0$ for every $x\in [0,1]$.

Fix $x\in [0,1]$.

If $x\not\in A_*$, then $\mu(\{x\})=0$ by \cref{lem42}, hence we assume now that $x\in A_*$ and choose an address of~$x$, that is a sequence $(x_k)_{k\in\mathbb N}$ of elements of $\{0,\ldots,N\}$ such that \eqref{x} holds. Note that the monotonicity of $f_0,\ldots,f_N$ implies
\begin{equation}\label{xx}
f_{x_1,\ldots,x_k}(0)\leq f_{x_1,\ldots,x_{k+1}}(0)\leq x\leq
f_{x_1,\ldots,x_{k+1}}(1)\leq f_{x_1,\ldots,x_k}(1)
\end{equation}
for every $k\in\mathbb N$.

First we want to prove that
\begin{equation}\label{f_n(0)}
\mu\big(\{f_n(0)\}\big)=\mu\big(\{f_n(1)\}\big)=0
\end{equation}
for every $n\in\{0,\ldots,N\}$.

We begin with proving that
\begin{equation}\label{01}
\mu(\{0\})=\mu(\{1\})=0.
\end{equation}
	
If $f_0(0)>0$, then~\eqref{inv} gives
\begin{equation*}
\mu(\{0\})=\sum_{n=0}^{N}p_n\mu(\{f_n^{-1}(0)\})=
\sum_{n=0}^{N}p_n\mu(\emptyset)=0.
\end{equation*}
If $f_0(0)=0$, then~\eqref{inv} yields
\begin{equation*}
\mu(\{0\})=\sum_{n=0}^{N}p_n\mu(\{f_n^{-1}(0)\})=
p_0\mu(\{0\})+\sum_{n=1}^{N}p_n\mu(\emptyset)=p_0\mu(\{0\}),
\end{equation*}
and since $p_0\in(0,1)$ we conclude that $\mu(\{0\})=0$.
	
In the same way, considering two cases ($f_N(1)<1$ and $f_N(1)=1$) and using~\eqref{inv} jointly with the fact that $p_N\in(0,1)$ in the second case, we get $\mu(\{1\})=0$.
	
Using~\eqref{inv}, \eqref{0<1} and \eqref{01} we obtain
\begin{equation*}
\mu(\{f_0(0)\})=\sum_{n=0}^{N}p_n\mu(\{f_n^{-1}(f_0(0))\})=p_0\mu(\{0\})=0,
\end{equation*}
\begin{equation*}
\mu(\{f_N(1)\})=\sum_{n=0}^{N}p_n\mu(\{f_n^{-1}(f_N(1))\})=p_N\mu(\{1\})=0
\end{equation*}
and
\begin{equation*}
\begin{split}
\mu(\{f_m(1),f_{m+1}(0)\})&= \sum_{n=0}^{N}p_n\mu(\{f_n^{-1}(f_m(1)),f_n^{-1}(f_{m+1}(0))\})\\
&\leq 2(p_m\mu(\{1\})+p_{m+1}\mu(\{0\}))=0
\end{split}
\end{equation*}
for every $m\in\{0,\ldots,N-1\}$.

By \eqref{f_n(0)}~and~\eqref{0<1}, equality~\eqref{inv} implies
\begin{equation}\label{mu(B)}
\mu(f_n(B))=p_n\mu(B)
\end{equation}
for all $n\in\{0,\ldots,N\}$ and Borel sets $B\subset[0,1]$.

Finally, applying \eqref{xx} and $k$ times equality~\eqref{mu(B)} jointly with the fact that $\mu$ is a probability measure we get
\begin{equation*}
\begin{split}
\mu(\{x\})&=\mu\left(\bigcap_{k\in\mathbb N} \big[f_{x_1,\ldots,x_k}(0), f_{x_1,\ldots,x_k}(1)\big]\right)\\
&=\lim_{k\to\infty}\mu\left(\big[f_{x_1,\ldots,x_k}(0), f_{x_1,\ldots,x_k}(1)\big]\right)\\
&=\lim_{k\to\infty}\prod_{i=1}^k p_{x_i}
\leq\lim_{k\to\infty}{\left(\max\{p_0,\ldots,p_N\}\right)}^k=0.
\end{split}
\end{equation*}
The proof is complete.
\end{proof}

Combining \cref{lem42,lem43} with~\eqref{A*} we get the following corollary.

\begin{corollary}\label{cor44}
The measure $\mu$ vanishes on each of the following intervals:
$[0,\0]$, $[\1,1]$, $[f_{n_1,\ldots,n_k}(\1),f_{n_1,\ldots,n_k+1}(\0)]$ with $k\in\mathbb N$, $n_1,\ldots,n_{k-1}\in\{0,\ldots,N\}$ and $n_k\in\{0,\ldots,N-1\}$.
\end{corollary}

Define the function $\varphi\colon[0,1]\to[0,1]$ by
\begin{equation*}
\varphi(x)=\mu([0,x]).
\end{equation*}
From now on the letter $\varphi$ will be reserved for the just defined function.

Repeating the proof of Theorem~3.3 from~\cite{MZ} we get the following result.

\begin{theorem}\label{thm41}
Either $\varphi\in\mathcal C_a$ or $\varphi\in\mathcal C_s$.
\end{theorem}

As a consequence of \cref{thm41}, we have $\varphi\in\mathcal C$. \cref{lem42} implies that $\varphi$ cannot be constant on an open interval having nonempty intersection with the attractor $A_*$. Therefore, all the constants $c_{n_1,\ldots,n_k}$  occurring in the assertion of \cref{lem32} (associated with the above constructed $\varphi$) are pairwise different and belong to the open interval $(0,1)$.

We finish this section by giving a precise formula for $\varphi$.

\begin{theorem}\label{formula}
Assume that $x\in[0,1]$.
\begin{enumerate}[label=(\roman*),ref=(\roman*)]
\item If $x\in[0,\0]$, then $\varphi(x)=0$.\label{phi is zero}
\item If $x\in[\1,1]$, then $\varphi(x)=1$.\label{phi is one}
\item If $x\in A_*$ and $(x_l)_{l\in\mathbb N}$ is an address of $x$, then
\begin{equation*}
\varphi(x)=\sum_{l=1}^{\infty}\sign(x_l)\left[\prod_{n=0}^{N}
p_n^{\#\{i\in\{1,\ldots,l-1\}:x_i=n\}}\cdot
\sum_{n=0}^{x_l-1}p_n\right].\label{formula for phi if address}
\end{equation*}	
\item If $x\in[f_{x_1,\ldots,x_k}(\1),f_{x_1,\ldots,x_k+1}(\0)]$ with  $k\in\mathbb N$, $x_1,\ldots,x_{k-1}\in\{0,\ldots,N\}$ and $x_k\in\{0,\ldots,N-1\}$, then
\begin{equation*}
\begin{split}
\varphi(x)=&\sum_{l=1}^{k}\sign(x_l)\left[ \prod_{n=0}^{N}p_n^{\#\{i\in\{1,\ldots,l-1\}:x_i=n\}}\cdot
\sum_{n=0}^{x_l-1}p_n\right]\\
&+\prod_{n=0}^{N}p_n^{\#\{i\in\{1,\ldots,k\}:x_i=n\}}.
\end{split}
\end{equation*}\label{formula for phi if not address}
\end{enumerate}
\end{theorem}

\begin{proof}
Assertions \cref{phi is zero} and \cref{phi is one} are trivially implied by assertions \eqref{restriction phi to 0} and \eqref{restriction phi to 1} of \cref{lem31}. The proof of assertion~\cref{formula for phi if address} follows very closely the proof of Theorem 3.6 from \cite{MZ}, so we omit it. Assertion~\cref{formula for phi if not address} is a consequence of \cref{lem32}, the fact that
$(x_1,\ldots,x_k,N,\ldots)$ is the address of the point $f_{x_1,\ldots,x_k}(\1)$ and assertion~\cref{formula for phi if address}; indeed
\begin{equation*}
\begin{split}
\varphi(x)&=\varphi(f_{x_1,\ldots,x_k}(\1))=\sum_{l=1}^{k}\sign(x_l)\left[ \prod_{n=0}^{N}p_n^{\#\{i\in\{1,\ldots,l-1\}:x_i=n\}}\cdot
\sum_{n=0}^{x_l-1}p_n\right]\\
&\hspace{3ex}+\sum_{l=k+1}^{\infty} \left[\prod_{n=0}^{N}p_n^{\#\{i\in\{1,\ldots,k\}:x_i=n\}}p_N^{l-k-1}\right]
\cdot\big(1-p_N\big)\\
&=\sum_{l=1}^{k}\sign(x_l)\left[ \prod_{n=0}^{N}p_n^{\#\{i\in\{1,\ldots,l-1\}:x_i=n\}}\cdot
\sum_{n=0}^{x_l-1}p_n\right]\\
&\hspace{3ex}+\prod_{n=0}^{N}p_n^{\#\{i\in\{1,\ldots,k\}:x_i=n\}}.
\end{split}
\end{equation*}	
This finishes the proof.
\end{proof}


\section{More about the class \texorpdfstring{$\mathcal C$}{𝒞}}\label{More}
As we have seen in \cref{thm41}, with each sequence $\big(p_0,\ldots,p_N\big)$ of positive real numbers satisfying~\eqref{prob} we have associated a continuous increasing surjective solution $\varphi_{p_0,\ldots,p_N}\colon[0,1]\to[0,1]$ of equation~\eqref{e}. We denote by $\mathcal W$ the set of all these solutions. The main purpose of this section is to prove the following result.

\begin{theorem}\label{thm51}
The set $\mathcal W$ is linearly independent and its convex hull is contained in $\mathcal C$.
\end{theorem}

\subsection{Proof of Theorem~\ref{thm51}}
The statement concerning the convex hull follows from \cref{lem21}.

The proof for the independence will be divided into several lemmas. Before we formulate the first one, note that for every $y\in A_*$ equality~\cref{structure of invariant set} guarantees that there exists at least one $n\in\{0,\ldots,N\}$ such that $y\in f_n(A_*)$. Therefore, we can define a transformation $T\colon A_*\to [0,1]$ by putting
\begin{equation*}
T(y)=f_{n(y)}^{-1}(y),
\end{equation*}
where
\begin{equation*}
n(y)=\max\{n\in\{0,\ldots,N\}:y\in f_n(A_*)\}.
\end{equation*}

\begin{lemma}\label{measure preserving}
The transformation $T$ maps $A_*$ into $A_*$ and it is measure preserving for $\mu$.
\end{lemma}

\begin{proof}
To see that $T(A_*)\subset A_*$ we fix $y\in A_*$. Then the injectivity of~$f_{n(y)}$ implies that there is exactly one $x\in A_*$ such that $y=f_{n(y)}(x)$. Thus $T(y)=x\in A_*$.
	
Now we prove that $T$ is measure preserving for $\mu$.

Fix a Borel set $B\subset A_*$. As $A_*\subset \bigcup_{n=0}^N [f_n(0),f_n(1)]$, we have
\begin{equation*}
T^{-1}(B)=\bigcup_{n=0}^N {\{y\in [f_n(0),f_n(1)]\cap A_*: T(y)\in B\}}.
\end{equation*}
Then using \cref{lem43} jointly with \eqref{0<1} and the fact that the set~$\bigcup_{i=0}^Nf_i^{-1}(y)$ contains just one element in the case where $y\in (f_n(0),f_n(1))$ and at most two elements in the case where $y\in\{f_n(0),f_n(1)\}$, we obtain
\begin{dmath*}
\mu(T^{-1}(B))=\sum_{n=0}^N\mu({\{y\in [f_n(0),f_n(1)]\cap A_*: T(y)\in B\}})
=\sum_{n=0}^N\mu({\{y\in [f_n(0),f_n(1)]\cap A_*: f_{n}^{-1}(y)\in B\}})
=\sum_{n=0}^N\mu({\{y\in [f_n(0),f_n(1)]\cap A_*: y\in f_n(B)\}}).
\end{dmath*}
Next note that $f_n(B)\subset f_n(A_*)\subset A_*$ for every $n\in\{0,\ldots,N\}$. Thus
\begin{dmath*}
\mu(T^{-1}(B))=\sum_{n=0}^N\mu(f_n(B)).
\end{dmath*}
Finally, according to \cref{mu(B)} we conclude that
\begin{dmath*}
\mu(T^{-1}(B))=\sum_{n=0}^{N}p_n\mu(B)=\mu(B),
\end{dmath*}
and the proof is complete.
\end{proof}

By \cref{attractor and addresses} the points in $A_*$ are exactly the ones that have an address. The next lemma shows that we might run into slight problems with the uniqueness of the addresses if
\begin{equation}\label{Nb}
f_0(0)=0,\quad f_N(1)=1\quad\hbox{ and }\quad N_b\neq\emptyset,
\end{equation}
where
\begin{equation*}
N_b=\{n\in \{0,\ldots,N-1\}: f_n(1)=f_{n+1}(0)\}.
\end{equation*}

\begin{lemma}\label{uniqueness}
$ $
\begin{enumerate}
\item\label{at most two addresses} Every point from $A_*$ has at most two addresses, and if a point from $A_*$ has two addresses, then \eqref{Nb} holds and exactly one of the addresses belongs to the set
\begin{equation*}
Z_b=\left\{(x_k)_{k\in\mathbb N}\in \{0,\ldots,N\}^{\mathbb N}:\exists_{n\in\mathbb N}(x_{n}\in N_b\hbox{ and }x_k=N\hbox{ for every }k>n)\right\}.
\end{equation*}
\item\label{additional address} If \eqref{Nb} holds and a point from $A_*$ has an address belonging to the set $Z_b$, then it also has an address not belonging to the set $Z_b$.
\item\label{exactly one address} Every point from $A_*$ has exactly one address if and only if \eqref{Nb} does not hold.
\end{enumerate}
\end{lemma}

\begin{proof}
(\ref{at most two addresses}) Assume that $(x_k)_{k\in\mathbb{N}}$ and $(y_k)_{k\in\mathbb{N}}$ are two different addresses of a point $x\in A_*$. Put
\begin{equation*}
m=\min\{k\in\mathbb{N}:x_k\neq y_k\}
\end{equation*}
and let $x_m<y_m$. Then according to \eqref{x} and \eqref{0<1} we have
\begin{dmath*}
x=f_{x_1,\ldots,x_{m}}\left(\lim_{k\to\infty}f_{x_{m+1},\ldots,x_{k}}(1)\right)
\leq f_{x_1,\ldots,x_m}(1)\leq f_{y_1,\ldots,y_m}(0)\leq f_{y_1,\ldots,y_m}\left(\lim_{k\to\infty} f_{y_{m+1},\ldots,y_{k}}(0)\right)=x.
\end{dmath*}
Thus $f_{x_1,\ldots,x_m}(1)=f_{y_1,\ldots,y_m}(0)$, and hence $f_{x_m}(1)=f_{y_m}(0)\in A_*$. Finally, making use of \eqref{0<1}, we conclude that $x_m\in N_b$, $f_0(0)=0$ and $f_N(1)=1$. In consequence \eqref{Nb} holds, $(x_k)_{k\in\mathbb{N}}\in Z_b$ and $(y_k)_{k\in\mathbb{N}}\not\in Z_b$. Moreover, if we assumed that $x$ has a third address $(z_k)_{k\in\mathbb{N}}$, different from both the first ones, we would have $(z_k)_{k\in\mathbb{N}}\in Z_b\setminus\{(x_k)_{k\in\mathbb{N}}\}$, which is impossible.

(\ref{additional address}) Assume that \eqref{Nb} holds and let a point $x\in A_*$ has an address $(x_k)_{k\in\mathbb{N}}\in Z_b$. Then there is $m\in N$ such that $x_m\in N_b$ and $x=f_{x_1,\ldots,x_m}(\1)$. Applying now \eqref{Nb} we get
\begin{equation*}
x=f_{x_1,\ldots,x_m}(1)=f_{x_1,\ldots,x_m+1}(0)=f_{x_1,\ldots,x_m+1}(\0),
\end{equation*}
which shows that \eqref{x} has an address not belonging to the set $Z_b$.

(\ref{exactly one address}) Assertion (\ref{at most two addresses}) implies that if \eqref{Nb} does not hold, then every point from $A_*$ has exactly one address.

Assume now that every point from $A_*$ has exactly one address and suppose that, on the contrary, \eqref{Nb} holds. Then $0=\0$ and $1=\1$. Fix $n\in N_b$ and put $x=f_n(\1)$. Then $x\in A_*$ and $f_n(\1)=f_n(1)=f_{n+1}(0)=f_{n+1}(\0)$, which jointly with \eqref{x} implies that $x$ has two different addresses, a contradiction.
\end{proof}

We now define a map $\pi\colon \{0,\ldots,N\}^{\field{N}}\setminus Z_b\to A_*$  by putting
\begin{equation*}
\pi((x_k)_{k\in\mathbb N})=\lim_{k\to\infty}f_{x_1,\ldots,x_k}(0),
\end{equation*}
where $Z_b$ is as in \cref{uniqueness} in the case where \eqref{Nb} holds and $Z_b=\emptyset$ in the case where \eqref{Nb} does not hold.

\begin{lemma}\label{pi is a bijection}
The map $\pi$ is a bijection.
\end{lemma}

\begin{proof}
It is enough to apply \cref{uniqueness} and \cite[Theorem 1 in Chapter 4.2]{B1988}.
\end{proof}

Denote by $\sigma$ the \emph{Bernoulli shift}, i.e.\ the map from $\{0,\ldots,N\}^{\field{N}}$ into itself defined by
\begin{equation*}
\sigma\big((x_k)_{k\in\mathbb N}\big)=(x_{k+1})_{k\in\mathbb N}.
\end{equation*}

\begin{lemma}\label{sigma pi T commute}
For every $n\in \field{N}$ we have
\begin{equation*}
\sigma^{-n}\circ \pi^{-1}=\pi^{-1}\circ T^{-n}.
\end{equation*}
\end{lemma}

\begin{proof}
We begin with proving that we have
\begin{equation}\label{diagram commutes}
\pi\circ\sigma=T\circ \pi
\end{equation}
on $\{0,\ldots,N\}^{\field{N}}\setminus Z_b$.

First, we note that $\sigma(\{0,\ldots,N\}^{\field{N}}\setminus Z_b)\subset\{0,\ldots,N\}^{\field{N}}\setminus Z_b$.
Fix $(x_k)_{k\in\mathbb N}\in \{0,\ldots,N\}^{\field{N}}\setminus Z_b$ and put
$z=\lim_{k\to\infty}f_{x_2,\ldots,x_k}(0)$.
Then
\begin{dmath*}
\pi(\sigma((x_k)_{k\in\mathbb N}))=\lim_{k\to\infty}f_{x_2,\ldots,x_k}(0)=z,
\end{dmath*}
and
\begin{dmath*}
T(\pi((x_k)_{k\in\mathbb N}))
=T\left(\lim_{k\to\infty}f_{x_1,\ldots,x_k}(0)\right)
=T\left(f_{x_1}(z)\right)=f_{n(f_{x_1}(z))}^{-1}(f_{x_1}(z)).
\end{dmath*}
Since $z\in A_*$, we have $f_{x_1}(z)\in f_{x_1}(A_*)$, and so
\begin{equation*}
x_1\leq n(f_{x_1}(z)).
\end{equation*}
Suppose towards a contradiction that $x_1<n(f_{x_1}(z))$. Then $f_{x_1}(z)\in f_{x_1+1}(A_*)$, and by \eqref{0<1} we have $z=1$. Therefore,
$f_{x_1}(1)=f_{x_1+1}(0)$ and $x_k=N$ for every $k\geq 2$, which is impossible as $(x_k)_{k\in\mathbb N}\not\in Z_b$.
In consequence $x_1=n(f_{x_1}(z))$. Hence $f_{n(f_{x_1}(z))}^{-1}(f_{x_1}(z))=z$, which yields that \eqref{diagram commutes} holds.

To complete the proof it is enough to proceed by induction with the use of \cref{diagram commutes}.
\end{proof}

Let us consider now the measure $\field{P}_{p_0,\ldots,p_N}$ defined on~$\{0,\ldots,N\}$  by
\begin{equation*}
\field{P}_{p_0,\ldots,p_N}(\{k\})=p_k.
\end{equation*}
Note that $\field{P}_{p_0,\ldots,p_N}$ is a probability measure by \eqref{prob}.
Further, we let $\field{P}$ be the product measure on $\{0,\ldots,N\}^{\field{N}}$ of $\mathbb N$ copies of the measure $\field{P}_{p_0,\ldots,p_N}$. It is known that the Bernoulli shift is strong-mixing for~$\field{P}$ (see e.g. \cite[Problem~4.3]{biswas2014ergodic} or \cite[Exercise 2.7.9]{EW2011}).

\begin{lemma}\label{nu is mupi}
If $B\subset\{0,\ldots,N\}^{\field{N}}\setminus Z_b$ is a Borel set, then
$\pi(B)$ is a Borel set and
\begin{equation}\label{Pmu}
\field{P}(B)=\mu(\pi(B)).
\end{equation}
\end{lemma}

\begin{proof}
We first prove that if $B\subset\{0,\ldots,N\}^{\field{N}}\setminus Z_b$ is a Borel set, then $\pi(B)$ is a Borel set as well.
	
As every Borel set in $\{0,\ldots,N\}^{\field{N}}\setminus Z_b$ is generated by sets of the form
\begin{equation}\label{cylinder}
B=\left(\{x_1\}\times \cdots\times \{x_m\}\times \{0,\ldots,N\}^{\field{N}}\right)\setminus Z_b,
\end{equation}
where $m\in\mathbb N$ and $x_1,\ldots,x_m\in\{0,\ldots,N\}$, it is sufficient to show that $\pi(B)$ is a Borel set for every set of the form \eqref{cylinder}.

Fix a set $B$ of the form \eqref{cylinder} with $m\in\mathbb N$ and $x_1,\ldots,x_m\in\{0,\ldots,N\}$. If either \eqref{Nb} does not hold or \eqref{Nb} holds and $x_m\not\in N_b$, then by \cref{pi is a bijection} we have
\begin{dmath}\label{piB1}
\pi(B)=f_{x_1,\ldots,x_m}\left(\pi\left(\{0,\ldots,N\}^{\field{N}}\setminus Z_b\right)\right)=f_{x_1,\ldots,x_m}(A_*).
\end{dmath}
If \eqref{Nb} holds and $x_m\in N_b$, then by \cref{pi is a bijection} we have
\begin{dmath}\label{piB2}
\pi(B)=f_{x_1,\ldots,x_m}\left(\pi\left((\{0,\ldots,N\}^{\field{N}}\setminus Z_b)\setminus \{N\}^{\field{N}}\right)\right)=f_{x_1,\ldots,x_m}(A_*)\setminus f_{x_1,\ldots,x_m}(\{\1\}).
\end{dmath}

Since $A_*$ and $\{\1\}$ are compact sets and $f_0,\ldots,f_N$ are contractions, it follows that $f_{x_1,\ldots,x_m}(A_*)$ and $f_{x_1,\ldots,x_m}(\{\1\})$ are compact sets. In consequence, we see that in both the considered cases the set $\pi(B)$ is Borel.
	
Now we prove that \eqref{Pmu} holds for every Borel set $B\subset\{0,\ldots,N\}^{\field{N}}\setminus Z_b$.
		
Since every two Borel probability measures defined on $\{0,\ldots,N\}^{\field{N}}\setminus Z_b$ agreeing on cylinders are equal, it is suffices to show that \eqref{Pmu} holds for every cylinder  $B\subset\{0,\ldots,N\}^{\field{N}}\setminus Z_b$. Moreover, by the additivity of the measures, we only need to show that \eqref{Pmu} holds for every set of the form \eqref{cylinder}, where $m\in\mathbb N$ and $x_1,\ldots,x_m\in\{0,\ldots,N\}$.

Fix a set $B$ of the form \eqref{cylinder} with $m\in\mathbb N$ and $x_1,\ldots,x_m\in\{0,\ldots,N\}$. Then either \eqref{piB1} or \eqref{piB2} is satisfied, and by \cref{lem43,lem42} we see that in both the cases we have
\begin{equation*}
\mu\left(f_{x_1,\ldots,x_m}^{-1}(\pi(B))\right)=\mu(A_*)=1.
\end{equation*}
This jointly with \cref{mu(B)} yields
\begin{dmath*}
\mu(\pi(B))=\mu\left(f_{x_1}\left(\cdots\left(f_{x_m}\left(f_{x_1,\ldots,x_m}^{-1}(\pi(B))\right)\right)\cdots\right)\right)=\prod_{i=1}^mp_{x_i}.
\end{dmath*}
Finally, note that $\field{P}(B)=\prod_{i=1}^mp_{x_i}$.
\end{proof}

\begin{lemma}\label{T is strong-mixing}
The transformation~$T$ is strong-mixing for~$\mu$.
\end{lemma}

\begin{proof}
The transformation $T$ is measure preserving for $\mu$ by \cref{measure preserving}. To prove that it is strong-mixing for~$\mu$, fix two Borel sets $A,B\subset A_*$. Then using \cref{pi is a bijection,nu is mupi}, the fact that the Bernoulli shift is strong-mixing for~$\field{P}$ and \cref{sigma pi T commute} we get
\begin{dmath*}
\mu(A)\mu(B)=\mu(\pi(\pi^{-1}(A)))\mu(\pi(\pi^{-1}(B)))
=\field{P}(\pi^{-1}(A))\field{P}(\pi^{-1}(B))
=\lim_{m\to\infty}\field{P}(\sigma^{-m}(\pi^{-1}(A))\cap\pi^{-1}(B))
=\lim_{m\to \infty}\field{P}(\pi^{-1}(T^{-m}(A))\cap \pi^{-1}(B))
=\lim_{m\to \infty}\field{P}(\pi^{-1}(T^{-m}(A)\cap B))
=\lim_{m\to \infty}\mu(T^{-m}(A)\cap B).
\end{dmath*}
The proof is complete.
\end{proof}

Denote by $\mathcal{M}^T(A_*)$ the set of all Borel probability measures defined on the $\sigma$-algebra of all Borel subset of the interval $[0,1]$ supported on $A_*$, making the transformation $T$ measure preserving. Note that $\mu\in \mathcal{M}^T(A_*)$ by \cref{lem42}.

\begin{lemma}\label{mutually singular then independent}
Every family of pairwise mutually singular measures belonging to the set $\mathcal{M}^T(A_*)$ is linearly independent.
\end{lemma}

\begin{proof}
Fix $m\in\mathbb N\setminus\{1\}$, pairwise mutually singular measures $\mu_1,\ldots,\mu_m\in \mathcal{M}^T(A_*)$, numbers $\alpha_1,\ldots,\alpha_m\in \mathbb R\setminus\{0\}$ and assume by contradiction that
\begin{equation*}\label{independent}
\sum_{i=1}^m \alpha_i\mu_i=0.
\end{equation*}	
Since the measures are mutually singular, for each $i,j\in\{1,\ldots,m\}$ with $i\not=j$ there are sets $A_i^j$ and $A_j^i=X\setminus A_i^j$ such that $\mu_i(A_i^i)=\mu_j(A_j^i)=0$. Put $A_m=\bigcup_{i=1}^{m-1}A_m^i$. Then
\begin{equation*}
0\leq \mu_m(A_m)=\mu_m\left(\bigcup_{i=1}^{m-1}A_m^i\right)\leq \sum_{i=1}^{m-1}\mu_m(A_m^i)=0
\end{equation*}
and for every $j\in\{1,\ldots,m-1\}$ we have
\begin{dmath*}
1\geq \mu_j(A_m)=\mu_j\left(X\setminus \bigcap_{i=1}^{m-1}A_i^m\right)
\geq \mu_j(X\setminus A_j^m)\geq \mu_j(X)-\mu_j(A_j^m)=1.
\end{dmath*}
In consequence
\begin{equation*}
\alpha_m=\sum_{i=1}^{m}\alpha_i\mu_i(X\setminus A_m)=0,
\end{equation*}
and the proof is complete.
\end{proof}

Now we are in a position to prove that the set $\mathcal{W}$ is linearly independent.

Fix $m$ different functions from $\mathcal{W}$ and consider the corresponding measures $\mu_1,\ldots,\mu_m\in \mathcal{M}^T(A_*)$. From \cref{T is strong-mixing} we infer that the transformation~$T$ is ergodic for all the measures. Thus $\mu_1,\ldots,\mu_m$ are extreme points of the set $\mathcal{M}^T(A_*)$ (see \cite[Theorem~4.4]{EW2011} or \cite[Proposition 12.4]{P2001}), and hence they are pairwise mutually singular. Invoking \cref{mutually singular then independent} gives the claim.

\subsection{An application of Theorem~\ref{thm51}}
In \cite{M1985} Janusz Matkowski posed a problem asking about the existence of nonlinear monotonic and continuous solutions $\Phi\colon[0,1]\to\mathbb R$ of a very particular case of the equation
\begin{equation}\label{e1}
\Phi(x)=\sum_{n=0}^{N}\Phi(f_n(x))-\sum_{n=1}^{N}\Phi(f_n(0)).
\end{equation}
Motivated by this problem, denote by $\mathcal M$ the vector space spanned by ${\mathcal W}\cup\{\mathds{1}\}$, where $\mathds{1}$ denotes the constant function that equals $1$ on $[0,1]$. Note that by \cref{thm51} and the fact that $\phi(0)=0$ for each $\phi\in\mathcal W$, the set ${\mathcal W}\cup\{\mathds{1}\}$ is a basis of $\mathcal M$.

\begin{prop}\label{prop52}
Every function belonging to $\mathcal M$ is a continuous solution of equation
\eqref{e1}. Moreover, if $\phi_1,\ldots,\phi_m\in{\mathcal W}$ and $\alpha_1,\ldots,\alpha_m\in\mathbb R$ are of the same sign, then the function $\sum_{i=1}^m\alpha_{i}\phi_{i}+\alpha_{0}$ is monotone for every $\alpha_{0}\in\mathbb R$.
\end{prop}

\begin{proof}
Fix $\Phi\in\mathcal M$. Then there exist $\alpha_0,\ldots,\alpha_m\in\mathbb R$ and $\phi_1,\ldots,\phi_m\in{\mathcal W}$ such that $\Phi=\sum_{i=1}^m\alpha_{i}\phi_{i}+\alpha_{0}$. Obviously, $\Phi$ is continuous. According to the first assertion of~\cref{lem22} we see that $\phi_i(f_0(0))=0$ for every \mbox{$i\in\{1,\ldots,m\}$}, and hence applying also \eqref{e}, we obtain
\begin{equation*}
\begin{split}
\sum_{n=0}^{N}\Phi(f_n(x))-\sum_{n=1}^{N}\Phi(f_n(0))&=
\sum_{n=0}^{N}\left(\sum_{i=1}^m\alpha_{i}\phi_{i}(f_n(x))+\alpha_{0}\right)\\
&\hspace*{3ex}-\sum_{n=1}^{N}\left(\sum_{i=1}^m\alpha_{i}\phi_{i}(f_n(0))+\alpha_{0}\right)\\
&=\sum_{i=1}^m\alpha_{i}\left(\sum_{n=0}^{N}\phi_{i}(f_n(x))-\sum_{n=0}^{N}\phi_{i}(f_n(0))\right)+\alpha_{0}\\
&=\sum_{i=1}^m\alpha_{i}\phi_{i}(x)+\alpha_{0}=\Phi(x)
\end{split}
\end{equation*}	
for every $x\in[0,1]$.
	
The moreover part of the assertion is clear.
\end{proof}

\section*{Acknowledgements} The research of both authors was supported by the Silesian University Mathematics Department (Iterative Functional Equations and Real Analysis program). Furthermore, the research leading to these results has received funding from the European Research Council under the European Union's Seventh Framework Programme (\ignoreSpellCheck{FP}/2007-2013) / \ignoreSpellCheck{ERC} Grant Agreement n.291497 while the second author was a postdoctoral researcher at the University of Warwick.

\bibliographystyle{plain}
\bibliography{Notes}

\end{document}